\documentclass[reqno]{amsart}
\usepackage{amssymb}
\usepackage{ifpdf}
\ifpdf
  \usepackage[hyperindex,pagebackref]{hyperref}
\else
  \expandafter\ifx\csname dvipdfm\endcsname\relax
    \usepackage[hypertex,hyperindex,pagebackref]{hyperref}
  \else
    \usepackage[dvipdfm,hyperindex,pagebackref]{hyperref}
  \fi
\fi
\theoremstyle{plain}
\newtheorem{thm}{Theorem}[section]
\newtheorem{lem}{Lemma}[section]
\numberwithin{equation}{section}
\DeclareMathOperator{\td}{d\mspace{-2mu}}
\allowdisplaybreaks[4]

\begin{document}

\title[some new Huygens and Wilker type inequalities]
{Refinements and sharpening of some Huygens and Wilker type inequalities}

\author[W.-D. Jiang]{Wei-Dong Jiang}
\address[Jiang]{Department of Information Engineering, Weihai Vocational University, Weihai City, Shandong Province, 264210, China}
\email{\href{mailto: W.-D. Jiang <jackjwd@163.com>}{jackjwd@163.com}}

\author[Q.-M. Luo]{Qiu-Ming Luo}
\address[Luo]{Department of Mathematics, Chongqing Normal University, Chongqing City, 401331, China}
\email{\href{mailto: Q.-M. Luo <luomath@126.com>}{luomath@126.com}, \href{mailto: Q.-M. Luo <luomath2007@163.com>}{luomath2007@163.com}}

\author[F. Qi]{Feng Qi}
\address[Qi]{School of Mathematics and Informatics\\ Henan Polytechnic University\\ Jiaozuo City, Henan Province, 454010\\ China; Department of Mathematics\\ School of Science\\ Tianjin Polytechnic University\\ Tianjin City, 300387\\ China}
\email{\href{mailto: F. Qi <qifeng618@gmail.com>}{qifeng618@gmail.com}, \href{mailto: F. Qi <qifeng618@hotmail.com>}{qifeng618@hotmail.com}, \href{mailto: F. Qi <qifeng618@qq.com>}{qifeng618@qq.com}}
\urladdr{\url{http://qifeng618.wordpress.com}}



\thanks{The first author was partially supported by the Project of Shandong Province Higher Educational Science and Technology Program under grant No.~J11LA57}

\begin{abstract}
In the article, some Huygens and Wilker type inequalities involving trigonometric and hyperbolic functions are refined and sharpened.
\end{abstract}

\thanks{This paper was typeset using \AmS-\LaTeX}

\maketitle

\section{Introduction}
The famous Huygens inequality for the sine and tangent functions states that for $x\in\bigl(0,\frac{\pi}{2}\bigr)$
\begin{equation}\label{eq1.1}
    2\sin x+\tan x>3x.
\end{equation}
The hyperbolic counterpart of \eqref{eq1.1} was established in~\cite{4} as follows: For $x>0$
\begin{equation}\label{eq1.2}
    2\sinh x+\tanh x>3x.
\end{equation}
The inequalities~\eqref{eq1.1} and~\eqref{eq1.2} were respectively refined in~\cite[Theorem~2.6]{4} as
\begin{equation}\label{eq1.3}
    2\frac{\sin x}{x}+\frac{\tan x}{x}>2\frac{x}{\sin x}+\frac{x}{\tan x}>3,\quad 0<x<\frac{\pi}{2}
\end{equation}
and
\begin{equation}\label{eq1.4}
    2\frac{\sinh x}{x}+\frac{\tanh x}{x}>2\frac{x}{\sinh x}+\frac{x}{\tanh x}>3,\quad x\ne0.
\end{equation}
In~\cite{11} the inequality~\eqref{eq1.2} was improved as
\begin{equation}\label{eq3.5}
    2\frac{\sinh x}{x}+\frac{\tanh x}{x}>3+\frac{3}{20}x^4-\frac{3}{56}x^6, \quad x>0.
\end{equation}
In~\cite{4}, the following inequality is given
\begin{equation}\label{eq3.7}
   3\frac{x}{\sin x}+\cos x>4.
\end{equation}
For more information in this area, please refer to~\cite{mia-qi-cui-xu-99, 2}, \cite[Section~1.7 and Section~7.3]{refine-jordan-kober.tex-JIA} and closely related references therein.
\par
In~\cite{12}, Wilker proved
\begin{equation}\label{w}
\biggl(\frac{\sin x}x\biggr)^2+\frac{\tan x}x>2
\end{equation}
and proposed that there exists a largest constant
$c$ such that
\begin{equation}\label{ww}
\biggl(\frac{\sin x}x\biggr)^2+\frac{\tan x}x>2+cx^3\tan x
\end{equation}
for $0<x<\frac{\pi}2$.
In~\cite{13}, the best constant $c$ in~\eqref{ww} was found and it was proved that
\begin{equation}\label{www}
2+\frac8{45}x^3\tan x>\biggl(\frac{\sin x}x\biggr)^2+\frac{\tan
x}x>2+\biggl(\frac2\pi\biggr)^4x^3\tan x
\end{equation}
for $0<x<\frac{\pi}2$. The constants $\frac8{45}$ and
$\bigl(\frac2\pi\bigr)^4$ in the inequality~\eqref{www} are the best possible.
For more information on this topic, please see~\cite{wilker-theory, wilker-mia, 9}, \cite[pp.~38\nobreakdash--40, Section~8]{refine-jordan-kober.tex-JIA} and closely related references therein.
\par
Recently the inequalities~\eqref{eq1.3} and~\eqref{w} were respectively refined in~\cite{7} as
\begin{equation}\label{eq1.7}
2\frac{\sin x}{x}+\frac{\tan x}{x}>\frac{\sin x}{x}+2\frac{\tan (x/2)}{x/2}> 2\frac{x}{\sin x}+\frac{x}{\tan x} >3
\end{equation}
and
\begin{multline}\label{eq1.8}
\biggl(\frac{\sin x}{x}\biggr)^2+\frac{\tan x}{x}> \biggl(\frac{x}{\sin x}\biggr)^2+\frac{x}{\tan x}\\*
>\frac{\sin x}{x}+\biggl[\frac{\tan (x/2)}{x/2}\biggr]^2 >\frac{x}{\sin x}+\biggl[\frac{x/2}{\tan (x/2)}\biggr]^2 >2.
\end{multline}
The hyperbolic counterparts of the last two inequalities in~\eqref{eq1.8} were also given in~\cite{7} as follows:
\begin{equation}\label{eq1.9}
   \frac{\sinh x}{x}+\biggl[\frac{\tanh (x/2)}{x/2}\biggr]^2
   >\frac{x}{\sinh x}+\biggl[\frac{x/2}{\tanh (x/2)}\biggr]^2  >2.
\end{equation}
\par
The aim of this paper is to refine and sharpen some of the above-mentioned Huygens and Wilker type inequalities.

\section{some Lemmas}

In order to attain our aim, we need several lemmas below.

\begin{lem}\label{B-2q-posit}
The Bernoulli numbers $B_{2n}$ for $n\in\mathbb{N}$ have the property
\begin{equation}\label{|B_{2n}|}
(-1)^{n-1}B_{2n}=|B_{2n}|,
\end{equation}
where the Bernoulli numbers $B_i$ for $i\ge0$ are defined by
\begin{equation}
\frac{x}{e^x-1}=\sum_{i=0}^\infty \frac{B_i}{n!}x^i =1-\frac{x}2+\sum_{i=1}^\infty B_{2i}\frac{x^{2i}}{(2i)!}, \quad \vert x\vert <2\pi.
\end{equation}
\end{lem}

\begin{proof}
In~\cite[p.~16 and p.~56]{aar}, it is listed that for $q\ge1$
\begin{equation}\label{zeta(2q)-B(2q)}
\zeta(2q)=(-1)^{q-1}\frac{(2\pi)^{2q}}{(2q)!}\frac{B_{2q}}2,
\end{equation}
where $\zeta$ is the Riemann zeta function defined by
\begin{equation}
  \zeta(s)=\sum_{n=1}^\infty\frac1{n^s}.
\end{equation}
From~\eqref{zeta(2q)-B(2q)}, the formula~\eqref{|B_{2n}|} follows.
\end{proof}

\begin{lem}\label{lem2.1}
For $0<|x|<\pi$, we have
\begin{equation}\label{eq2.1}
    \frac{x}{\sin x}=1+\sum_{n=1}^{\infty} \frac{2\bigl(2^{2n-1}-1\bigr)|B_{2n}|}{(2n)!}x^{2n}.
\end{equation}
\end{lem}

\begin{proof}
This is an easy consequence of combining the equality
\begin{equation}\label{4.3.68}
\frac1{\sin x}=    \csc x=\frac1x +\sum_{n=1}^{\infty} \frac{(-1)^{n-1}2\bigl(2^{2n-1}-1\bigr)B_{2n}}{(2n)!}x^{2n-1},
\end{equation}
see~\cite[p.~75, 4.3.68]{abram}, with Lemma~\ref{B-2q-posit}.
\end{proof}

\begin{lem}[{\cite[p.~75, 4.3.70]{abram}}]\label{lem2.0}
For $0<|x|<\pi$,
\begin{equation}\label{eq2.0}
   \cot x=\frac{1}{x}-\sum_{n=1}^{\infty}\frac{2^{2n}|B_{2n}|}{(2n)!}x^{2n-1}.
\end{equation}
\end{lem}

\begin{lem}\label{lem2.2}
For $0<|x|<\pi$,
\begin{equation}\label{eq2.2}
\frac{1}{\sin ^2x}=\frac{1}{x^2} +\sum_{n=1}^{\infty}\frac{2^{2n}(2n-1)|B_{2n}|}{(2n)!}x^{2(n-1)}.
\end{equation}
\end{lem}

\begin{proof}
Since
$$
\frac{1}{\sin ^2x}=\csc^2x=-\frac{\td}{\td x}(\cot x),
$$
the formula~\eqref{eq2.2} follows from differentiating~\eqref{eq2.0}.
\end{proof}

\begin{lem}\label{lem2.3}
For $0<|x|<\pi$,
\begin{equation}\label{eq2.3}
\frac{\cos x}{\sin ^2x}=\frac{1}{x^2}-\sum_{n=1}^{\infty} \frac{2(2n-1)\bigl(2^{2n-1}-1\bigr)|B_{2n}|}{(2n)!}x^{2(n-1)}.
\end{equation}
\end{lem}

\begin{proof}
This follows from differentiating on both sides of~\eqref{4.3.68} and using~\eqref{|B_{2n}|}.
\end{proof}

\begin{lem}\label{lem2.4}
For $0<|x|<\pi$,
\begin{multline}\label{eq2.4}
\frac{1}{\sin^3 x}=\frac1{2x}+\frac1{x^3} \\* +\frac12\sum_{n=1}^\infty\frac1{(2n-1)!}\biggl[\frac{\bigl(2^{2n+1}-1\bigr)|B_{2n+2}|}{n+1} +\frac{\bigl(2^{2n-1}-1\bigr)|B_{2n}|}n\biggr]x^{2n-1}
\end{multline}
and
\begin{equation}\label{eq2.5}
\frac{\cos x}{\sin^3x}=\frac{1}{x^3}- \sum_{n=2}^{\infty}\frac{(2n-1)(n-1)2^{2n}|B_{2n}|}{(2n)!}x^{2n-3}.
\end{equation}
\end{lem}

\begin{proof}
Combining
\begin{equation*}
\frac{1}{\sin ^3 x}=\frac{1}{2\sin x}-\frac12\biggl(\frac{\cos x}{\sin^2x}\biggr)'
\end{equation*}
with Lemma~\ref{lem2.3}, the identity~\eqref{4.3.68}, and Lemma~\ref{B-2q-posit} gives~\eqref{eq2.4}.
\par
The equality~\eqref{eq2.5} follows from combination of
\begin{equation*}
\frac{\cos x}{\sin ^3 x}=-\frac12\biggl(\frac{1}{\sin^2x}\biggr)'
\end{equation*}
with Lemma~\ref{lem2.2}.
\end{proof}

\begin{lem}\label{lem2.5}
Let $f$ and $g$ be continuous on $[a,b]$ and differentiable in $(a,b)$ such that $g'(x)\ne0$ in $(a,b)$. If $\frac{f'(x)}{g'(x)}$ is increasing $($or decreasing$)$ in $(a,b)$, then the functions $\frac{f(x)-f(b)}{g(x)-g(b)}$ and $\frac{f(x)-f(a)}{g(x)-g(a)}$ are also increasing $($or decreasing$)$ in
$(a,b)$.
\end{lem}

The above Lemma~\ref{lem2.5} can be found, for examples, in~\cite[p.~292, Lemma~1]{6}, \cite[p.~57, Lemma~2.3]{jordan-generalized-simp.tex}, \cite[p.~92, Lemma~1]{Gene-Jordan-Inequal.tex}, and~\cite[p.~161, Lemma~2.3]{jordan-strengthened.tex}.

\begin{lem}\label{lem2.6}
Let $a_k$ and $b_k$ for $k\in\mathbb{N}$
be real numbers and the power series
\begin{equation}
A(x)=\sum_{k=1}^\infty a_kx^k\quad\text{and}\quad B(x)=\sum_{k=1}^\infty
b_kx^k
\end{equation}
be convergent on $(-R,R)$ for some $R>0$. If $b_k>0$ and the ratio $\frac{a_k%
}{b_k}$ is strictly increasing for $k\in\mathbb{N}$, then the function $%
\frac{A(x)}{B(x)}$ is also strictly increasing on $(0,R)$.
\end{lem}

The above Lemma~\ref{lem2.6} can be found, for examples, in~\cite[p.~292, Lemma~2]{6}, \cite[p.~71, Lemma~1]{elliptic-mean-comparison-rev2.tex}, and~\cite[Lemma~2.2]{5}.

\section{Main results}

Now we are in a position to state and prove our main results, refinements and sharpening of some Huygens and Wilker type inequalities mentioned in the first section.

\begin{thm}\label{th3.1}
For $|x|\in\bigl(0,\frac\pi2\bigr)$, we have
\begin{equation}\label{eq3.1}
3+\frac{1}{60}x^3\sin x<2\frac{x}{\sin x}+\frac{x}{\tan x}<3+\frac{8\pi-24}{\pi^3}x^3\sin x.
\end{equation}
The scalars $\frac{1}{60}$ and $\frac{8\pi-24}{\pi^3}$ in~\eqref{eq3.1} are the best possible.
\end{thm}

\begin{proof}
Let
\begin{equation*}
f(x)=\frac{\frac{2x}{\sin x}+\frac{x}{\tan x}-3}{x^3\sin x}
=\frac{\frac{2x}{\sin ^2x}+\frac{x\cos x}{\sin ^2x}-\frac3{\sin x}}{x^3}
\end{equation*}
for $x\in\bigl(0,\frac\pi2\bigr)$.
By virtue of~\eqref{eq2.1}, \eqref{eq2.2}, and \eqref{eq2.3}, we have
\begin{align*}
f(x)&=\frac{1}{x^3}\Biggl[\sum_{n=1}^{\infty}\frac{2^{2n+1}(2n-1)}{(2n)!}|B_{2n}|x^{2n-1}
-\sum_{n=1}^{\infty}\frac{\bigl(2^{2n}-2\bigr)}{(2n)!}|B_{2n}|x^{2n-1}\\
&\quad -\sum_{n=1}^{\infty}\frac{\bigl(2^{2n}-2\bigr)(2n-1)}{(2n)!}|B_{2n}|x^{2n-1}\Biggr]\\
&=\sum_{n=1}^{\infty}\frac{2^{2n+1}(2n-1)-(2n-1) \bigl(2^{2n}-2\bigr)-3\bigl(2^{2n}-2\bigr)}{(2n)!}|B_{2n}|x^{2n-4}\\
&=\sum_{n=2}^{\infty}\frac{(n-2)2^{2n+1}+4(n+1)}{(2n)!}|B_{2n}|x^{2n-4}.
\end{align*}
So the function $f(x)$ is strictly increasing on $\bigl(0,\frac\pi2\bigr)$. Moreover, it is easy to obtain
\begin{equation*}
\lim_{x\to0^+}f(x)=\frac1{60}\quad \text{and}\quad \lim_{x\to(\pi/2)^-}f(x)=\frac{8\pi-24}{\pi^3}.
\end{equation*}
The proof of Theorem~\ref{th3.1} is complete.
\end{proof}

\begin{thm}\label{th3.2}
For $0<|x|< \pi/2$,
\begin{equation}\label{eq3.2}
2+\frac{17}{720}x^3\sin x<\frac{x}{\sin x}+\biggl[\frac{x/2}{\tan (x/2)}\biggr]^2 <2+\frac{\pi^2+8\pi-32}{2\pi^3}x^3\sin x.
\end{equation}
The constants $\frac{17}{720}$ and $\frac{\pi^2+8\pi-32}{2\pi^3}$ in~\eqref{eq3.2} are the best possible.
\end{thm}

\begin{proof}
By using \eqref{eq2.1}, \eqref{eq2.2}, \eqref{eq2.4}, and~\eqref{eq2.5}, the function
\begin{align*}
g(x)&=\frac{\frac{x}{\sin x}+\bigl[\frac{x/2}{\tan (x/2)}\bigr]^2-2}{x^3\sin x}\\
&=\frac{1}{4}\frac{2x^2\cos x+4x\sin x+2x^2-x^2\sin^2 x-8\sin^2x}{x^3\sin^3 x}\\
&=\frac{1}{4x^3}\biggl(\frac{2x^2\cos x}{\sin ^3 x}+\frac{4x}{\sin^2x}+\frac{2x^2}{\sin ^3 x} -\frac{x^2}{\sin x}-\frac{8}{\sin x}\biggr)
\end{align*}
may be expanded as
\begin{align*}
g(x)&=\frac{1}{4x^3}\Biggl[\sum_{n=1}^{\infty}\frac{2^{2n+2}(2n-1)|B_{2n}|}{(2n)!}x^{2n-1} -\sum_{n=2}^{\infty}\frac{2^{2n+1}(2n-1)(n-1)|B_{2n}|}{(2n)!}x^{2n-1}\\
   &\quad+\sum_{n=2}^{\infty}\frac{\bigl(2^{2n}-2\bigr)(2n-1)(2n-2)|B_{2n}|}{(2n)!}x^{2n-1} +\sum_{n=1}^{\infty}\frac{\bigl(2^{2n}-2\bigr)|B_{2n}|}{(2n)!}x^{2n+1}\\
   &\quad-\sum_{n=1}^{\infty}\frac{\bigl(2^{2n}-2\bigr)|B_{2n}|}{(2n)!}x^{2n+1} -\sum_{n=1}^{\infty}\frac{8\bigl(2^{2n}-2\bigr)|B_{2n}|}{(2n)!}x^{2n-1}\Biggr]\\
      &=\frac{1}{4}\sum_{n=2}^{\infty}\frac{(2n-1)2^{2n+2} -8\bigl(2^{2n}-2\bigr)-2(2n-1)(2n-2)}{(2n)!}|B_{2n}|x^{2n-4}\\
   &= \sum_{n=2}^{\infty}\frac{4^{n}(2n-3)+3+3n-2n^2}{(2n)!}|B_{2n}|x^{2n-4}\\
   &\triangleq\sum_{n=2}^{\infty}\frac{b_n}{(2n)!} |B_{2n}|x^{2n-4}.
\end{align*}
Since $b_2=17$ and
\begin{multline*}
b_{n+1}-b_n=4^n(6n-1)-4n+1
=(1+3)^n(6n-1)-4n+1\\
>3n(6n-1)-4n+1
=18n(n-2)+29(n-2)+59
>0
\end{multline*}
for $n\ge 2$, the sequence $b_n$ is increasing and $b_n\ge b_2=17>0$. Thus, the function $g(x)$ is increasing on $\bigl(0,\frac{\pi}{2}\bigr)$. Moreover,
\begin{equation*}
\lim_{x\to 0^+} g(x)=\frac{17}{720}\quad \text{and}\quad  \lim_{x\to(\pi/2)^-}g(x)=\frac{\pi^2+8\pi-32}{2\pi^3}.
\end{equation*}
The proof of Theorem~\ref{th3.2} is complete.
\end{proof}

\begin{thm}\label{th3.3}
For $x>0$, we have
\begin{equation}\label{eq3.3}
2\frac{\sinh x}{x}+\frac{\tanh x}{x}>3+\frac{3}{20}x^3\tanh x.
\end{equation}
The constant $\frac{3}{20}$ is the best possible.
\end{thm}

\begin{proof}
Let
\begin{equation*}
F(x)=\frac{\frac{2\sinh x}{x}+\frac{\tanh x}{x}-3}{x^3\tanh x}
=\frac{\sinh 2x+\sinh x-3x\cosh x}{x^4\sinh x}
\end{equation*}
and let
\begin{equation*}
f(x)=\sinh 2x+\sinh x-3x\cosh x\quad \text{and}\quad g(x)=x^4\sinh x.
\end{equation*}
From the power series expansions
\begin{equation}\label{sinh-cosh-expansion}
\sinh x=\sum_{n=0}^{\infty}\frac{x^{2n+1}}{(2n+1)!}\quad \text{and}\quad
\cosh x=\sum_{n=0}^{\infty}\frac{x^{2n}}{(2n)!},
\end{equation}
it follows that
\begin{align*}
f'(x)&= 2\cosh 2x-2\cosh x-3x\sinh x\\ &=\sum_{n=0}^{\infty}\frac{2^{2n+1}x^{2n}}{(2n)!} -\sum_{n=0}^{\infty}\frac{2x^{2n}}{(2n)!} -\sum_{n=0}^{\infty}\frac{3x^{2n+2}}{(2n+1)!}\\
&= \sum_{n=0}^{\infty}\frac{\bigl(2^{2n+1}-2\bigr)x^{2n}}{(2n)!} -\sum_{n=0}^{\infty}\frac{3x^{2n+2}}{(2n+1)!}\\
   &=\sum_{n=1}^{\infty}\frac{\bigl(2^{2n+1}-2\bigr)x^{2n}}{(2n)!} -\sum_{n=1}^{\infty}\frac{6nx^{2n+2}}{(2n)!}\\
&=\sum_{n=2}^{\infty}\frac{\bigl(2^{2n+1}-6n-2\bigr)x^{2n}}{(2n)!}\\
&\triangleq\sum_{n=2}^{\infty}a_nx^{2n}
\end{align*}
and
\begin{multline*}
g'(x)=4x^3\sinh x+x^4\cosh x
= \sum_{n=0}^{\infty}\frac{4x^{2n+4}}{(2n+1)!}+\sum_{n=0}^{\infty}\frac{x^{2n+4}}{(2n)!}\\
= \sum_{n=0}^{\infty}\frac{(2n+5)x^{2n+4}}{(2n+1)!}
=\sum_{n=1}^{\infty}\frac{(2n+3)x^{2n+2}}{(2n-1)!}\\
=\sum_{n=2}^{\infty}\frac{4n(n-1)\bigl(4n^2-1\bigr)x^{2n}}{(2n)!}
\triangleq\sum_{n=2}^{\infty}b_nx^{2n}.
\end{multline*}
It is easy to see that the quotient
\begin{equation*}
c_n=\frac{a_n}{b_n}=\frac{2^{2n+1}-6n-2}{4n(n-1)\bigl(4n^2-1\bigr)}
\end{equation*}
satisfies
\begin{equation*}
    c_{n+1}-c_n=\frac{\bigl(6n^2-17n+1\bigr)4^n+18n^2+23n-1}{2n(2n+3)\bigl(4n^2-1\bigr)(n^2-1)}>0
\end{equation*}
for $n\ge2$. This means that the sequence $c_n$ is increasing. By Lemma~\ref{lem2.6}, the function $G(x)=\frac{f'(x)}{g'(x)}$ is increasing on $(0,\infty)$, and so, by Lemma~\ref{lem2.5}, the function $F(x)=\frac{f(x)}{g(x)}=\frac{f(x)-f(0)}{g(x)-g(0)}$ is increasing on $(0,\infty)$. Moreover, it is not difficult to obtain $\lim_{x \to 0^+}F(x)=c_2=\frac{3}{20}$. Theorem~\ref{th3.3} is thus proved.
\end{proof}

\begin{thm}\label{th3.4}
For $x>0$,
\begin{equation}\label{eq3.6}
\frac{\sinh x}{x}+\biggl[\frac{\tanh(x/2)}{x/2}\biggr]^2> 2+\frac{23}{720}x^3\tanh x.
\end{equation}
The number $\frac{23}{720}$ in~\eqref{eq3.6} is the best possible.
\end{thm}

\begin{proof}
Let
\begin{align*}
F(x)&=\frac{\frac{\sinh x}{x}+\bigl[\frac{\tanh(x/2)}{x/2}\bigr]^2-2}{x^3\tanh x}\\
&=\frac{\cosh x\bigl(x\sinh x\cosh x+x\sinh x+4\cosh x -2x^2\cosh x-4-2x^2\bigr)}{x^5\sinh x(1+\cosh x)}
\end{align*}
and let
\begin{equation*}
f(x)=\cosh x\bigl(x\sinh x\cosh x+x\sinh x+4\cosh x -2x^2\cosh x-4-2x^2\bigr)
\end{equation*}
and
\begin{equation*}
g(x)=x^5\sinh x(1+\cosh x).
\end{equation*}
By the power series expansions in~\eqref{sinh-cosh-expansion}, we obtain
\begin{align*}
f(x)&=x\sinh x+\frac{1}{4}x\sinh(3x)-\frac{3}{4}x\sinh x+\frac{1}{2}\sinh(2x)-4\cosh x\\
 &\quad+2\cosh(2x)-x^2\cosh(2x)-2x^2\cosh x+2-x^2\\
&=\sum_{n=0}^{\infty}\frac{\frac{3^{2n+1}}{4}-n2^{2n+1} -4n-\frac{7}{4}}{(2n+1)!}x^{2n+2}+\sum_{n=0}^{\infty}\frac{2^{2n+1}-4}{(2n)!}x^{2n}+2-x^2\\
&=\sum_{n=1}^{\infty}\frac{\frac{3^{2n+1}}{4}-n2^{2n+1} -4n-\frac{7}{4}}{(2n+1)!}x^{2n+2}+\sum_{n=2}^{\infty}\frac{2^{2n+1}-4}{(2n)!}x^{2n}\\
&=\sum_{n=2}^{\infty}\frac{\frac{3^{2n-1}}{4}-(n-1)2^{2n-1} -4n+\frac{9}{4}}{(2n-1)!}x^{2n}+\sum_{n=2}^{\infty}\frac{2^{2n+1}-4}{(2n)!}x^{2n}\\
&=\sum_{n=3}^{\infty}\frac{n\bigl[\frac{3^{2n-1}}{2}-(n-1)2^{2n} -8n+\frac{9}{2}\bigr]+2^{2n+1}-4}{(2n)!}x^{2n}\\
&\triangleq\sum_{n=3}^{\infty}a_nx^{2n}
\end{align*}
and
\begin{multline*}
    g(x)=x^5\biggl[\frac{1}{2}\sinh(2x)+\sinh x\biggr]
    =\sum_{n=0}^{\infty}\frac{1+2^{2n}}{(2n+1)!}x^{2n+6}
    =\sum_{n=3}^{\infty}\frac{1+2^{2n-6}}{(2n-5)!}x^{2n}\\
    =\sum_{n=3}^{\infty}\frac{\bigl(1+2^{2n-6}\bigr)(2n-4)(2n-3)(2n-2)(2n-1)2n}{(2n)!}x^{2n}
    \triangleq\sum_{n=3}^{\infty}b_nx^{2n}.
\end{multline*}
The ratio
\begin{equation*}
c_n=\frac{a_n}{b_n}=\frac{n\bigl[\frac{3^{2n-1}}{2}-(n-1)2^{2n} -8n+\frac{9}{2}\bigr]+2^{2n+1}-4}{\bigl(1+2^{2n-6}\bigr)(2n-4)(2n-3)(2n-2)(2n-1)2n}
\end{equation*}
satisfies
\begin{align*}
c_3&=\frac{23}{720}=0.031\dotsc, & c_4&=\frac{17}{336}=0.050\dotsc,& &\text{and} & c_5&=\frac{5099}{85680}=0.059\dotsc.
\end{align*}
Furthermore, when $n\ge 6$,  by a simple computation, we have
\begin{equation*}
    c_{n+1}-c_n=\frac{f_1(n)+f_2(n)+f_3(n)+f_4(n)}{3n(16+4^n)(64+4^n)(n-2)(2n-3)\bigl(4n^2-1\bigr)(n^2-1)},
\end{equation*}
where
\begin{align*}
    f_1(n)&=16^n\bigl(144n^3-24n^2-648n+240\bigr)\\
    &=16^n\bigl[144 n (n-6)^2+1704 n (n-6)+4392 (n-6)+26592\bigr]\\
    &>0,\\
    f_2(n)&=9^n\bigl(1024n^3-3072n^2-640n+3456\bigr)\\
    &=9^n\bigl[1024 n (n-6)^2+9216 n (n-6)+128 (139 n+27)\bigr]\\
    &>0,\\
    f_4(n)&=18432n^3+7680n^2-3456n-15744\\
    &=18432 n (n-6)^2+228864 n (n-6)+384 (1839 n-41)\\
    &>0,
\end{align*}
and
\begin{align*}
f_3(n)&=4^n\bigl[9^n\bigl(10n^3-57n^2-13n+54\bigr)\\
&\quad-2016n^4+8622n^3+5541n^2-33327n+17490\bigr]\\
&=4^n\bigl\{9^n\bigl[10 n (n-6)^2+63 n (n-6)+5 (n-6)+84\bigr]\\
&\quad-2016n^4+8622n^3+5541n^2-33327n+17490\bigr\}\\
&>4^n\bigl[84(1+8)^n-2016n^4+8622n^3+5541n^2-33327n+17490\bigr]\\
&>4^n\biggl\{84\biggl[1+8\binom{n}{1}+8^2\binom{n}{2}+8^3\binom{n}{3}+8^4\binom{n}{4}\biggr]\\
&\quad-2016n^4+8622n^3+5541n^2-33327n+17490\biggr\}\\
&=17574 - 107023 n + 144421 n^2 - 70226 n^3 + 12320 n^4\\
&=12320 (n-6)^4+225454 (n-6)^3+1541473 (n-6)^2\\
&\quad+4686101(n-6)+5372496\\
&>0
\end{align*}
for $n\ge6$. Hence, the sequence $c_n$ is increasing. By Lemma~\ref{lem2.6}, the function $F(x)$ is increasing.
Finally, it is easy to see that $\lim_{x \to 0^+}F(x)=c_3=\frac{23}{720}$.
The proof of Theorem~\ref{th3.4} is complete.
\end{proof}

\begin{thm}\label{th3.5}
For $0<|x|< \pi/2$, we have
\begin{equation}\label{eq3.8}
4+\frac{1}{10}x^3\sin x<3\frac{x}{\sin x}+\cos x<4+\frac{12\pi-32}{\pi^3}x^3\sin x.
\end{equation}
The numbers $\frac{1}{10}$ and $\frac{12\pi-32}{\pi^3}$ are the best possible.
\end{thm}

\begin{proof}
Let
\begin{equation*}
f(x)=\frac{3\frac{x}{\sin x}+\cos x-4}{x^3\sin x}
=\frac{1}{x^3}\biggl(\frac{3x}{\sin^2x}+\cot x-\frac{4}{\sin x}\biggr).
\end{equation*}
By \eqref{eq2.1}, \eqref{eq2.0}, and \eqref{eq2.2}, we have
\begin{align*}
   f(x)&=\frac{1}{x^3}\Biggl[\frac{3}{x}-\sum_{n=1}^{\infty}\frac{3(2n-1)2^{2n}}{(2n)!}|B_{2n}|x^{2n-1} +\frac{1}{x}\\
   &\quad-\sum_{n=1}^{\infty}\frac{2^{2n}}{(2n)!}|B_{2n}|x^{2n-1}-\frac{4}{x} -\sum_{n=1}^{\infty}\frac{4\bigl(2^{2n}-2\bigr)}{(2n)!}|B_{2n}|x^{2n-1}\Biggr]\\
   &=\sum_{n=1}^{\infty}\frac{3(2n-1)2^{2n}-2^{2n}-4\bigl(2^{2n}-2\bigr)}{(2n)!}|B_{2n}|x^{2n-4}\\
  &=\sum_{n=1}^{\infty}\frac{(6n-8)2^{2n}+8}{(2n)!}|B_{2n}|x^{2n-4}\\
   &=\sum_{n=2}^{\infty}\frac{(6n-8)2^{2n}+8}{(2n)!}|B_{2n}|x^{2n-4}.
\end{align*}
This shows that the function $f(x)$ is increasing on $\bigl(0,\frac{\pi}{2}\bigr)$. Moreover, it is straightforward to obtain
\begin{equation*}
\lim_{x\to 0^+} f(x)=a_2=\frac{1}{10} \quad\text{and}\quad
\lim_{x\to (\pi/2)^-}f(x)=\frac{12\pi-32}{\pi^3}.
\end{equation*}
The proof of Theorem~\ref{th3.5} is complete.
\end{proof}

\end{document}